\documentclass[11pt]{amsart}
\usepackage{mathrsfs}
\usepackage{amssymb}
\usepackage[all]{xy}
\usepackage{amsthm}
\usepackage{array}
\usepackage{amsmath}
\usepackage{enumitem}

\newtheorem{theorem}{Theorem}[section]
\newtheorem{lemma}[theorem]{Lemma}
\newtheorem{proposition}[theorem]{Proposition}
\newtheorem{corollary}[theorem]{Corollary}

\newtheorem{definition}[theorem]{Definition}

\newtheorem{conjecture}[theorem]{Conjecture}

\def\be{\begin{equation}}
\def\ee{\end{equation}}
\def\br{\begin{eqnarray}}
\def\er{\end{eqnarray}}

 \title{Strongly Gauduchon spaces}
 \author{Lingxu Meng}
 \author{Wei Xia}
 \address{Department of Mathematics\\  North University of China \\ Taiyuan, Shanxi 030051,
P.R. China}
\email{20160012@nuc.edu.cn}
\address{Institute of Mathematics\\  Fudan University \\ Shanghai
200433, P.R. China}
\email{11110180005@fudan.edu.cn}

\thanks{}
 \date{}

\begin{document}
\maketitle

\begin{abstract}
We define strongly Gauduchon spaces and the class $\mathscr{SG}$ which are generalization of strongly Gauduchon manifolds in complex spaces. Comparing with the case of K\"ahlerian, the strongly Gauduchon space and the class $\mathscr{SG}$ are similar to the K\"ahler space and the Fujiki class $\mathscr{C}$ respectively. Some properties about these complex spaces are obtained, and the relations between the strongly Gauduchon spaces and the class $\mathscr{SG}$ are studied.
\end{abstract}

\textbf{keywords}: strongly Gauduchon metric, strongly Gauduchon space, \\ class $\mathscr{SG}$, topologically essential map.

\textbf{AMSC}: 32C15, 32C10, 53C55,

\section{Introduction}
The complex manifold with a strongly Gauduchon metric is an important object in non-K\"ahler geometry. In \cite{Pop1}, D. Popovici first defined the strongly Gauduchon metric in the study of limits of projective manifolds under deformations.  A \emph{strongly Gauduchon} metric on a complex $n$-dimensional manifold is a hermitian metric $\omega$ such that $\partial\omega^{n-1}$ is $\overline{\partial}$-exact. A compact complex manifold is called a \emph{strongly Gauduchon manifold}, if there exists a strongly Gauduchon metric on it.
\begin{proposition}
 Let $M$ be a compact complex manifold of dimension $n$. Then the following is equivalent.
\begin{enumerate}[leftmargin=*]
\item[(1)] $M$ is a strongly Gauduchon manifold.
\item[(2)] There exists a strictly positive $(n-1, n-1)$-form $\Omega$, such that $\partial\Omega$ is $\overline{\partial}$-exact.
\item[(3)] There exists a real closed $(2n-2)$-form $\Omega$ whose $(n-1, n-1)$-component $\Omega^{n-1, n-1}$ is strictly positive.
\end{enumerate}
\end{proposition}
In \cite{Pop1}, D. Popovici observed (1) and (3) are equivalent.  ``$(1)\Rightarrow (2)$" is obvious by the definition of strongly Gauduchon manifolds. Conversely, for any strictly positive $(n-1, n-1)$-form $\Omega$, there exists a unique strictly positive $(1, 1)$-form $\omega$, such that $\omega^{n-1}=\Omega$ (see \cite{M}, page 280). So we have ``$(2)\Rightarrow (1)$".

D. Popovici proved following two important theorems.
\begin{theorem}[\cite{Pop1}, Proposition 3.3]\label{sGc}
Let $M$ be a compact complex manifold. Then $M$ is a strongly Gauduchon manifold  if and only if there is no nonzero positive current $T$ of bidegree $(1, 1)$ on $M$ which is $d$-exact on $M$.
\end{theorem}

\begin{theorem}[\cite{Pop2}, Theorem 1.3]\label{msG}
Let $f:M\rightarrow N$ be a modification of  compact complex manifolds. Then $M$ is a strongly Gauduchon manifold if and only if $N$ is a strongly Gauduchon manifold.
\end{theorem}

On the other hand, in \cite{Fu1}, A. Fujiki generalized the concept ``K\"ahler" to general complex spaces. A kind of generalization is the \emph{K\"ahler space} which is a complex space admitting a strictly positive closed $(1, 1)$-form and the other kind is the \emph{Fujiki class $\mathscr{C}$} consisting of the reduced compact complex spaces  which are the meoromorphic images of a compact K\"ahler spaces. In \cite{V1} and \cite{V2}, J. Varouchas proved that any reduced complex space in the Fujiki class $\mathscr{C}$ has a proper modification which is a compact K\"ahler manifold. Now, many authors use it as the definition of the Fujiki class $\mathscr{C}$. Inspired by the method of A. Fujiki and the theorem of J. Varouchas, we give two kinds of generalization of strongly Gauduchon manifolds to complex spaces: the strongly Gauduchon spaces and  class $\mathscr{SG}$. In view of definitions of them, the strongly Gauduchon spaces (see Definition \ref{def sG}) is similar to the K\"ahler spaces and the class $\mathscr{SG}$ (see Definition \ref{classSG}) is similar to the Fujiki class $\mathscr{C}$.

In section $2$, we study the properties of strongly Gauduchon spaces and give a method of constructing examples which are singular strongly Gauduchon spaces, but not in $\mathscr{B}$, where $\mathscr{B}$ is the set of reduced compact complex spaces which are bimeromorphic to compact balanced manifolds.

In section $3$, we study the class $\mathscr{SG}$ and propose a conjecture on the relation between strongly Gauduchon spaces and the class $\mathscr{SG}$ as follows.
\begin{conjecture}
Any strongly Gauduchon space belongs to class $\mathscr{SG}$.
\end{conjecture}
We prove it in some special cases (see Theorem \ref{relation}, \ref{betti}, \ref{exc-relation}).

In section $4$, we study a family of reduced complex spaces over a nonsingular curve and give a theorem on the total space being in $\mathscr{SG}$.\\

\section{Strongly Gauduchon spaces}

First, we give a proposition about strongly Gauduchon manifolds which is similar to the case of balanced manifolds.
\begin{proposition}\label{sm, p}
Let $M$ and $N$ be compact complex manifolds of pure dimension.
\begin{enumerate}[leftmargin=*]
\item[(1)] If $f:M\rightarrow N$ is a holomorphic submersion  and $M$ is a strongly Gauduchon manifold, then $N$ is a strongly Gauduchon manifold.
\item[(2)] $M \times N$ is a strongly Gauduchon manifold, if and only if,  $M$ and $N$ are both strongly Gauduchon manifolds.
\end{enumerate}
\end{proposition}
\begin{proof}
Set $\dim M=m$, $\dim N=n$.

(1) Let $\Omega_M$ be a strictly positive $(m-1, m-1)$-form, such that $\partial\Omega_M=\overline{\partial}\alpha$, where $\alpha$ is a $(2m-2)$-form on $M$. Define
\begin{displaymath}
\Omega_N:=f_*\Omega_M.
\end{displaymath}
 By the proof of Proposition 1.9(ii) in \cite{M}, we know $\Omega_N$ is a strictly positive $(n-1, n-1)$-form. Obviously, $\partial\Omega_N=\overline{\partial}(f_*\alpha)$ is $\overline{\partial}$-exact. So $N$ is a strongly Gauduchon manifold.

(2) If $M\times N$ is a strongly Gauduchon manifold, then $M$ and $N$ are both strongly Gauduchon manifolds by (i).

Conversely, let $M$ and $N$ be both strongly Gauduchon manifolds. Suppose $\omega_M$ and $\omega_N$ are strongly Gauduchon metrics on $M$ and $N$ respectively, such that $\partial\omega_M^{m-1}=\overline{\partial}\alpha$ and $\partial\omega_N^{n-1}=\overline{\partial}\beta$, where $\alpha$ and $\beta$ are $(2m-2)$ and $(2n-2)$-form on $M$ and $N$ respectively. We define a metric on $M\times N$
\begin{displaymath}
\omega:=\omega_M+\omega_N
\end{displaymath}
then
\begin{displaymath}
\omega^{m+n-1}:=C_1\omega_M^{m-1}\wedge\omega_N^n+C_2\omega_M^m\wedge\omega_N^{n-1},
\end{displaymath}
where $C_1, C_2$ are constants. So
\begin{eqnarray*}
\partial\omega^{m+n-1}:&=&C_1\partial\omega_M^{m-1}\wedge\omega_N^n+C_2\omega_M^m\wedge\partial\omega_N^{n-1}   \\
&=&\overline{\partial}(C_1\alpha\wedge\omega_N^n+C_2\omega_M^m\wedge\beta)
\end{eqnarray*}
is $\overline{\partial}$-exact on $M\times N$. Hence $\omega$ is a strongly Gauduchon metric on $M\times N$.
\end{proof}

We recall the definitions of forms and currents on complex spaces, following \cite{K}.

Let $X$ be a reduced complex space and $X_{reg}$ the set of nonsingular points on $X$. Obviously, $X_{reg}$ is a complex manifold.

Suppose that $X$ is  an analytic subset of a complex manifold $M$. Set $I_X^{p, q}(M)= \{\alpha\in A^{p, q}(M) \mid i^*\alpha= 0\}$, where $i: X_{reg}\rightarrow M$ is the inclusion. Define $A^{p, q}(X):= A^{p, q}(M)/I_X^{p, q}(M)$. It can be easily shown that $A^{p, q}(X)$ does not depend on the embedding of $X$ into $M$. Hence, for any complex space $X$, we can define $A^{p, q}(X)$ through the local embeddings in $\mathbb{C}^N$. More precisely, we define a sheaf of germs $\mathcal{A}^{p, q}_X$ of $(p, q)$-forms on $X$ and $A^{p, q}(X)$ as the group of its global sections. Similarly, we can also define $A_c^{p, q}(X)$ (the space of $(p, q)$-forms with compact supports), $A^k(X)$ and $A_c^k(X)$.

We can naturally define $\partial:A^{p, q}(X)\rightarrow A^{p+1, q}(X)$,  $\overline{\partial}:A^{p, q}(X)\rightarrow A^{p, q+1}(X)$ and $d:A^k(X)\rightarrow A^{k+1}(X)$.

If $f: X\rightarrow Y$ is a holomorphic map between reduced complex spaces, then we can naturally define $f^*:A^{p, q}(Y)\rightarrow A^{p, q}(X)$ such that $f^*$ commutes with $\partial$, $\overline{\partial}$, $d$.

When $X$ is a subvariety of a complex manifold $M$, we define the space of currents on $X$
\begin{displaymath}
\mathcal{D}'^r(X):= \{T\in \mathcal{D}'^r(M)\mid T(u)= 0, \forall u\in I_{X, c}^{2n-r}(M)\},
\end{displaymath}
where $\mathcal{D}'^r(M)$ is the space of currents on $M$ and $I_{X, c}^{2n-r}(M)= \{\alpha\in A_c^{2n-r}(M) \mid i^*\alpha= 0\}$. We can define a space $\mathcal{D}'^r(X)$ of the currents on any reduced complex space $X$ as the case of $A^r(X)$. Define
\begin{displaymath}
\mathcal{D}'^{p, q}(X):= \{T\in \mathcal{D}'^{p+q}(X) \mid T(u)= 0, \forall u\in A_c^{r, s}(M), (r, s)\neq (n-p,n-q)\}.
\end{displaymath}
A current $T$ is called a \emph{$(p, q)$-current} on $X$, if $T\in\mathcal{D}'^{p, q}(X)$.
If $T\in \mathcal{D}'^r(X)$, we call $r$ the \emph{degree}. If $T\in \mathcal{D}'^{p, q}(X)$, we call $(p, q)$ the \emph{bidegree}. We also denote $ \mathcal{D}_r'(X)=\mathcal{D}'^{2n-r}(X)$ and $\mathcal{D}_{p, q}'(X)=\mathcal{D}'^{n-p, n-q}(X)$. A current $T\in \mathcal{D}'^{p, p}(X)$ is called \emph{real} if for every $\alpha\in A_c^{2n-2p}(X)$, $T(\overline{\alpha})= \overline{T(\alpha)}$.

If $f: X\rightarrow Y$ is a holomorphic map of reduced compact complex spaces, we define $f_*:\mathcal{D}_r'(X)\rightarrow \mathcal{D}_r'(Y)$ as $f_*T(u):= T(f^*u)$ for any  $u\in A_c^r(Y)$.

A real $(p, p)$-form $\omega$ on $X$ is called \emph{strictly positive}, if there exist an open covering $\mathcal{U}= \{U_\alpha\}$ of $X$ with an embedding $i_\alpha:U_\alpha\rightarrow V_\alpha$ of $U_\alpha$ into a domain $V_\alpha$ in $\mathbb{C}^{n_\alpha}$ and a strictly positive $(p, p)$-form $\omega_\alpha$ on $V_\alpha$, such that $\omega\mid_{U_\alpha}= i_\alpha^*\omega_\alpha$, for each $\alpha$.

Now, we give a kind of generalization of strongly Gauduchon manifolds.
\begin{definition}\label{def sG}
A purely $n$-dimensional reduced compact complex space $X$ is called a strongly Gauduchon space, if there exists a strictly positive $(n-1, n-1)$-form $\Omega$, such that $\partial\Omega$ is $\overline{\partial}$-exact.
\end{definition}
By its definition, it is easy to see that $X$ is a strongly Gauduchon space, if and only if, there exists a real closed $(2n-2)$-form $\Omega'$ on $X$ whose $(n-1, n-1)$-component $\Omega'^{n-1, n-1}$ is strictly positive. Indeed, if $\Omega$ is a strictly positive $(n-1, n-1)$-form , such that $\partial\Omega=\overline{\partial}\alpha$, where $\alpha$ is a $(n, n-2)$-form, then
\begin{displaymath}
\Omega':=\Omega-\alpha-\bar{\alpha}
\end{displaymath}
is the desired form. Conversely, since $\Omega'$ is real and $d$-closed, $\partial\Omega'^{n-1, n-1}=-\overline{\partial}\Omega'^{n, n-2}$. Hence, $\Omega:=\Omega'^{n-1, n-1}$ is the desired form.

Obviously, strongly Gauduchon manifolds and compact balanced spaces are strongly Gauduchon spaces.

\begin{proposition}
Let $X$ be a reduced compact complex space of pure dimension and $M$ a compact complex manifold of pure dimension. If $X\times M$ is a strongly Gauduchon space, then $M$ is a strongly Gauduchon manifold.
\end{proposition}
\begin{proof}
Let $X_{reg}$ be the set of nonsingular points on $X$ and $\Omega$ a strictly positive $(n+m-1, n+m-1)$-form on $X\times M$, such that $\partial\Omega$ is $\overline{\partial}$-exact, where $n=\dim X$ and $m=\dim M$.  Suppose $\pi: X_{reg}\times M\rightarrow M$ is the second projection. By the proof of Proposition 1.9(ii) in \cite{M}, we know $\pi_*(\Omega\mid_{X_{reg}\times M})$ is a strictly positive $(m-1, m-1)$-form on $M$. Obviously, $\partial\pi_*(\Omega\mid_{X_{reg}\times M})$ is $\overline{\partial}$-exact. So $M$ is a strongly Gauduchon manifold.
\end{proof}

We know that, on a compact balanced manifold $M$, the fundamental class $[V]$ of any hypersurface $V$ is not zero in $H^2(M, \mathbb{R})$ (see \cite{M}, Corollary 1.7). It is equivalent to that, the current $[V]$ on $M$ defined by any hypersurface $V$ is not $d$-exact. For strongly Gauduchon spaces, we have following proposition.
\begin{proposition}
If $X$ is a strongly Gauduchon space, then the current $[V]$ defined by any hypersurface $V$ of $X$ is not $\partial\overline{\partial}$-exact.
\end{proposition}
\begin{proof}
Suppose $\dim X=n$. Let $\Omega$ be a strictly positive $(n-1, n-1)$-form on $X$ such that $\partial\Omega=\overline{\partial}\alpha$, where $\alpha$ is a $(2n-2)$-form on $X$. If $[V]=\partial\overline{\partial}Q$ for some current $Q$ on $X$, then
\begin{displaymath}
[V](\Omega)=\int_V\Omega>0.
\end{displaymath}
On the other hand,
\begin{displaymath}
[V](\Omega)=(\partial\overline{\partial}Q)(\Omega)=-Q(\overline{\partial}\partial\Omega)=-Q(\overline{\partial}\overline{\partial}\alpha)=0.
\end{displaymath}
It is a contradiction.
\end{proof}

\begin{proposition}\label{loc}
If $f:X\rightarrow Y$ is a finite holomorphic unramified covering map of reduced compact complex spaces of pure dimension, then $X$ is a strongly Gauduchon space if and only if $Y$ is a strongly Gauduchon space.
\end{proposition}
\begin{proof}
Set $n=\dim X=\dim Y$ and $d=\deg f$.

Let $X$ be a strongly Gauduchon space and $\Omega_X$ a strictly positive $(n-1, n-1)$-form on $X$ such that $\partial\Omega_X=\overline{\partial}\alpha_X$, where $\alpha_X$ is a $2(n-1)$-form on $X$. For every $y\in Y$, we set $f^{-1}(y)=\{x_1, ..., x_d\}$, then there exists an open neighbourhood $V\subseteq Y$ of $y$, and open neighbourhoods $U_1$, $...$, $U_d$ of $x_1, ..., x_d$ in $X$ respectively, which do not intersect with each other, such that $f^{-1}(V)= \cup_{i=1}^d {U_i}$ and the restriction $f|_{U_i}:U_i\rightarrow V$ is an isomorphism for $i=1,...,d$. We define two forms on $V$ as
\begin{displaymath}
\Omega_V:=\Sigma_{i=1}^d(f|_{U_i}^{-1})^*(\Omega_X|_{U_i})
\end{displaymath}
\begin{displaymath}
\alpha_V:=\Sigma_{i=1}^d(f|_{U_i}^{-1})^*(\alpha_X|_{U_i})
\end{displaymath}
If $V$ and $V'$ are two open subsets in Y as above (possible for different points in $Y$) and $V\cap V'\neq \emptyset$, we can easily check $\Omega_V= \Omega_{V'}$ on $V\cap V'$. Hence we can construct a global $(n-1, n-1)$-form $\Omega_Y$ on $Y$ such that $\Omega_Y|_V= \Omega_V$. By the same reason, we can define a global $2(n-1)$-form $\alpha_Y$ on $Y$ such that $\alpha_Y|_V= \alpha_V$.  Obviously, $\Omega_Y$ is strictly positive and $\partial\Omega_Y=\overline{\partial}\alpha_Y$. Therefore, $Y$ is a strongly Gauduchon space.

Conversely, suppose $\Omega_Y$ is a strictly positive $(n-1, n-1)$-form on $Y$, such that $\partial\Omega_Y$ is $\overline{\partial}$-exact on $Y$. For all $x\in X$, there is an open neighbourhood $U$ of $x$ in $X$, an open neighbourhood $V$ of $f(x)$ in $Y$, such that $f\mid_U:U\rightarrow V$ is an isomorphism.  $(f^*\Omega_Y)|{U}=(f|_{U})^*(\Omega_Y|_{V})$ is obviously strictly positive on $U$, so is $f^*\Omega_Y$ on $X$. Obviously, $f^*\Omega_Y$ is $\overline{\partial}$-exact on $X$.  Therefore, $X$ is a strongly Gauduchon space.
\end{proof}

\section{The class $\mathscr{SG}$}
Now, we give the other generalization of strongly Gauduchon manifolds.
\begin{definition}\label{classSG}
A reduced  compact complex space $X$ of pure dimension is called in class $\mathscr{SG}$, if it has a desingularization $\widetilde{X}$ which is a strongly Gauduchon manifold.
\end{definition}
If one desingularization of $X$ is a strongly Gauduchon manifold, then every desingularization of $X$ is a strongly Gauduchon manifold. Indeed, if $X_1\to{X}$ and $X_2\to X$ are two desingularizations of $X$, then there exists a bimeromorphic map  $f: X_1\dashrightarrow X_2$. Let $\Gamma\subseteq X_1\times X_2$ be the graph of $f$, and $p_1:\Gamma\rightarrow X_1$, $p_2:\Gamma\to X_2$ the two projections on $X_1$, $X_2$, respectively. Then $p_1$, $p_2$ are modifications. If $\widetilde{\Gamma}$ is a desingularization of $\Gamma$, then $\widetilde{\Gamma}\to X_1$ and $\widetilde{\Gamma}\to X_2$ are modifications of compact complex manifolds. By Theorem \ref{msG}, we know that $X_1$ is  a strongly Gauduchon manifold if and only if $\widetilde{\Gamma}$ is a strongly Gauduchon manifold, and then if and only if $X_2$ is a strongly Gauduchon manifold. Hence Definition 2.1 is not dependent on the choice of the desingularization of $X$. So, if $X\in \mathscr{SG}$ is nonsingular, then $X$ is a strongly Gauduchon manifold.

Using the same method as above, we can prove the following proposition.

\begin{proposition}\label{bSG}
 The class $\mathscr{SG}$ is invariant under bimeromorphic maps.
\end{proposition}

Obviously, strongly Gauduchon manifolds and the normalizations of complex spaces in class $\mathscr{SG}$ are in class $\mathscr{SG}$. Recall that a reduced compact complex space $X$ is called \emph{in class $\mathscr{B}$}, if it has a desingulariztion  $\widetilde{X}$ which is a balanced manifold, referring to \cite{FMX}. Then complex spaces in class $\mathscr{B}$ are in class $\mathscr{SG}$.

\begin{proposition}\label{sGp}
If $X$ and $Y$ are reduced compact complex spaces, then $X\times Y$ is in the class $\mathscr{SG}$ if and only if $X$ and $Y$ are both in the class $\mathscr{SG}$.
\end{proposition}
\begin{proof}
If $f:\widetilde{X}\rightarrow X$ and $g:\widetilde{Y}\rightarrow Y$ are desingulariztions, then $f\times g:\widetilde{X}\times \widetilde{Y}\rightarrow X\times Y$ is a desingulariztion of $X\times Y$. By Proposition \ref{sm, p}(ii), we know that $\widetilde{X}\times\widetilde{Y}$ is a strongly Gauduchon manifold if and only if $\widetilde{X}$ and $\widetilde{Y}$ are both strongly Gauduchon manifolds. So we get this proposition easily.
\end{proof}

Using this proposition, we can construct some examples of complex spaces in $\mathscr{SG}$ which are neither strongly Gauduchon manifolds nor in class $\mathscr{B}$. If $Y$ is a singular reduced compact complex space in class $\mathscr{B}$ and $Z$ is a compact strongly Gauduchon manifold but not a balanced manifold, then $Y\times Z$ is in $\mathscr{SG}$, but it is neither a strongly Gauduchon manifold nor in $\mathscr{B}$. Indeed, $Y\times Z$ is singular, so it is not a  strongly Gauduchon manifold. By Proposition \ref{sGp}, $Y\times Z\in \mathscr{SG}$. Assume $Y\times Z\in \mathscr{B}$, by \cite{FMX}, Proposition 2.3, we know $Z\in \mathscr{B}$. Since $Z$ is nonsingular, $Z$ is balanced, which contradicts the choice of $Z$. Hence we get the following relations
\begin{displaymath}
\mathscr{C}\subsetneqq \mathscr{B}\subsetneqq \mathscr{SG},
\end{displaymath}
where $\mathscr{C}$ is the Fujiki class and the first ``$\subsetneqq$" is proved in \cite{FMX}, Section 2 .

If $X$ is a reduced compact complex space of pure dimension, then $X \in\mathscr{SG}$ if and only if every irreducible component of $X$ is in $\mathscr{SG}$. Indeed, if let  $\widetilde{X_1}$, $\ldots$, $\widetilde{X_r}$ be the desingulariztions of $X_1$, $\ldots$, $X_r$,  all the irreducible components of $X$, then the disjoint union $\widetilde{X}$:=$\widetilde{X_1} \amalg \ldots \amalg \widetilde{X_r}$ is a desingulariztion of $X$. Hence the conclusion follows since $\widetilde{X}$ is a strongly Gauduchon manifold if and only if $\widetilde{X_1}$, $\ldots$, $\widetilde{X_r}$ are all strongly Gauduchon manifolds.

In the following, we need the definition of a smooth morphism, referring to \cite{D}, (0.4). A surjective holomorphic map $f:X\rightarrow Y$ between reduced complex spaces is called a \emph{smooth} morphism, if for all $x\in X$, there is an open neighbourhood $W$ of $x$ in $X$, an open neighbourhood $U$ of $f(x)$ in $Y$, such that $f(W)= U$ and there is a commutative diagram
\begin{displaymath}
\xymatrix{
  W \ar[d]_{g} \ar[r]^{f|_W} &  U      \\
  \Delta^r\times U \ar[ur]_{pr_2}     }
\end{displaymath}
where $r=\dim X-\dim Y$, $g$ is an isomorphism (i.e., biholomorphic map), $pr_2$ is the second projection, and $\Delta^r$ is a small polydisc. Moreover, if $\dim X=\dim Y$, a smooth morphism is exactly a \emph{surjective local isomorphism}.

Obviously, if $f:X\rightarrow Y$ is a smooth morphism and $Y$ is a complex manifold, then $X$ must also be a complex manifold and $f$ is a submersion between complex manifolds.

\begin{proposition}\label{sm2}
Let $f:X\rightarrow Y$ be a smooth morphism of  reduced compact complex spaces. If $X \in\mathscr{SG}$, then $Y \in\mathscr{SG}$.
\end{proposition}
\begin{proof}
 Suppose $p:\widetilde{Y}\rightarrow Y$ is a desingulariztion. Consider the following Cartesian diagram
\begin{equation}\label{Car}
\xymatrix{
\widetilde{X}:=X\times_Y \widetilde{Y} \ar[d]_{q} \ar[r]^-{\tilde{f}}  & \widetilde{Y} \ar[d]^{p}  \\
  X \ar[r]_{f}  & Y},
\end{equation}
where $X\times_Y \widetilde{Y}=\{(x, \tilde y)\in X\times \widetilde{Y}| f(x)=p(\tilde y)\}$, $q$ is the  projection to $X$, and $\tilde{f}$ is the  projection to $\widetilde{Y}$. We can prove that $\tilde{f}$ is a submersion of complex manifolds and $q$ is a modification, referring to \cite{FMX}, Claim 1 and 2 in the proof of Proposition 2.4. Since $X \in\mathscr{SG}$, $\widetilde{X}$ is a strongly Gauduchon manifold, so is $\widetilde{Y}$ by Proposition \ref{sm, p} (i), hence $Y\in \mathscr{SG}$.
\end{proof}

\begin{proposition}
If $f:X\rightarrow Y$ is a finite unramified covering map of reduced compact complex spaces, then $X\in\mathscr{SG}$ if and only if $Y\in\mathscr{SG}$.
\end{proposition}
\begin{proof}
 Suppose $p:\widetilde{Y}\rightarrow Y$ is a desingulariztion. Consider the Cartesian diagram (\ref{Car}). We know that $\tilde{f}$ is a surjective local isomorphism and $q$ is a modification. Since $\widetilde{Y}$ is locally compact, by \cite{Ho}, Lemma 2, $\tilde f$ is a finite covering map in topological sense. Moreover, since $\tilde{f}$ is a local isomorphism (in analytic sense),  $\tilde f$ is a finite unramifield covering map (in analytic sense). By Proposition \ref{loc}, we know $\widetilde{X}$ is a strongly Gauduchon manifold, if and only if $\widetilde{Y}$ is a strongly Gauduchon manifold. Hence $X\in\mathscr{SG}$ if and only if $Y\in\mathscr{SG}$.
\end{proof}

We generalize Theorem 3.5 (2) and Theorem 3.9 (2) in \cite{A} as follows.
\begin{proposition}
Let $f:X\rightarrow Y$ be a smooth morphism of  reduced compact complex spaces, and $n=\dim X>m=\dim Y\geq2$. If $Y\in \mathscr{B}$ and there exists a point $y_0$ in $Y$ such that the current $[f^{-1}(y_0)]$ is not d-exact on $X$, then $X \in\mathscr{SG}$.
\end{proposition}
\begin{proof}
Choose a desingulariztion $p:\widetilde{Y}\rightarrow Y$ such that $\widetilde{Y}$ is a compact balanced manifold. Considering the Catesian diagram (\ref{Car}), we know that $\tilde{f}$ is a submersion of complex manifolds and $q$ is a modification.

For every $\tilde{y}\in p^{-1}(y_0)$, the current $[\tilde{f}^{-1}(\tilde{y})]$ can not be written as $dQ$ for any current $Q$ of degree $2m-1$ on $\widetilde{X}$. If not, since $\tilde{f}^{-1}(\tilde{y})= f^{-1}(y_0)\times \{\tilde{y}\}$,  we have
\begin{displaymath}
[f^{-1}(y_0)]= q_*[\tilde{f}^{-1}(\tilde{y})]=q_*(dQ)=dq_*Q,
\end{displaymath}
which contradicts the assumption.

Now suppose $\tilde{y}'$ is any point in $\widetilde{Y}$. Then the fundamental classes $[\tilde{y}]=[\tilde{y}']$ in $H^{2m}(\widetilde{Y} , \mathbb{R})$. Since $\tilde{f}$ is smooth,
\begin{displaymath}
[\tilde{f}^{-1}(\tilde{y}')]= \tilde{f}^*[\tilde{y}']= \tilde{f}^*[\tilde{y}]=[\tilde{f}^{-1}(\tilde{y})]
\end{displaymath}
in $H^{2m}(\widetilde{X}, \mathbb{R})$, where $\tilde{y}\in p^{-1}(y_0)$ and $\tilde{f}^*: H^{2m}(\widetilde{Y}, \mathbb{R})\rightarrow H^{2m}(\widetilde{X}, \mathbb{R})$ is the pull back of $\tilde{f}$. Hence for every $\tilde{y}'\in \widetilde{Y}$, the current $[\tilde{f}^{-1}(\tilde{y}')]$ is not $d$-exact on $\widetilde{X}$. By \cite{A}, Theorem 3.5 (2) and Theorem 3.9 (2), $\widetilde{X}$ is a strongly Gauduchon manifold, hence $X \in\mathscr{SG}$.
\end{proof}

Next we consider the relation between strongly Gauduchon spaces and class $\mathscr{SG}$. From definitions of them, the relation between strongly Gauduchon spaces and class $\mathscr{SG}$ is similar to that of K\"ahler spaces and  Fujiki class $\mathscr{C}$. Moreover, in the nonsingular case, we know that a modification of a strongly Gauduchon manifold is also a strongly Gauduchon manifold, by Theorem \ref{msG}. So we think the following also hold.
\begin{conjecture}\label{conj}
Any strongly Gauduchon space belongs to class $\mathscr{SG}$.
\end{conjecture}
We can prove it in some extra conditions. First, we recall a theorem and several notations.

\begin{theorem}[\cite{AB2}, Theorem 1.5]\label{support}
 Let $M$ be a complex manifold of dimension $n$, $E$ a compact analytic subset and $\{E_i\}_{i=1,...,s}$ all the $p$-dimensional irreducible components of $E$. If $T$ is a $\partial\overline{\partial}$-closed positive $(n-p, n-p)$-current on $M$ such that $supp T \subseteq E$, then there exist constants $c_i\geq 0$ such that $T- \Sigma_{i=1}^s c_i [E_i]$ is supported on the union of the irreducible components of $E$ of dimension greater than $p$.
\end{theorem}

For a compact complex manifold $M$, the \emph{Bott-Chern cohomology group} of degree $(p, q)$ is defined as
\begin{displaymath}
H^{p,q}_{BC}(M):= \frac{Ker(d: A^{p, q}(M)\rightarrow A^{p+q+1}(M))}{\partial\overline{\partial}A^{p-1, q-1}(M)}.
\end{displaymath}
and the \emph{Aeplli cohomology group} of degree $(p, q)$ is defined as
\begin{displaymath}
H^{p,q}_{A}(M):= \frac{Ker(\partial\overline{\partial}: A^{p, q}(M)\rightarrow A^{p+1, q+1}(M))}{\partial A^{p-1, q}(M) + \overline{\partial}A^{p, q-1}(M)}.
\end{displaymath}
It is well known that all these groups can also be defined  by means of currents of corresponding degree. For every $(p, q)\in \mathbb{N}^2$, the identity induces a natural map
\begin{displaymath}
i: H^{p,q}_{BC}(M)\rightarrow H^{p,q}_{A}(M).
\end{displaymath}
In general, the map $i$ is neither injective nor surjective. If $M$ satisfies $\partial\overline{\partial}$-lemma, then for every $(p, q)\in \mathbb{N}^2$, $i$ is an isomorphism, referring to \cite{DGMS}, Lemma 5.15, Remarks 5.16, 5.21.

\begin{theorem}\label{relation}
Let $X$ be a strongly Gauduchon space. If it has a desingularization $\widetilde{X}$ such that $i: H^{1, 1}_{BC}(\widetilde{X})\rightarrow H^{1,1}_{A}(\widetilde{X})$ is injective, then $X\in \mathscr{SG}$.
\end{theorem}
\begin{proof}
  Set $\dim X= n$. Suppose  $\pi: \widetilde{X}\rightarrow X $ is the desingularization. We need to prove that $\widetilde{X}$ is a strongly Gauduchon manifold. By Theorem \ref{sGc}, it suffices to prove that if $T$ is a positive $(1, 1)$-current on $\widetilde{X}$ which is $d$-exact, then $T= 0$.

Let $E\subseteq\widetilde{X}$ be the exceptional set of $\pi$, $\Omega$ the real closed $(2n-2)$-form on $X$ whose $(n-1, n-1)$-part $\Omega^{n-1, n-1}$ is strictly positive. Since $T$ is $d$-exact, we have $T(\pi^*\Omega)=0$. On the other hand, since $T$ is a $(1, 1)$-current,  we have
\begin{displaymath}
T(\pi^*\Omega)= T(\pi^*\Omega^{n-1, n-1})= \int_{\widetilde{X}} T\wedge \pi^*\Omega^{n-1, n-1}
\end{displaymath}
and $\pi^*\Omega^{n-1, n-1}$ is strictly positive on $\widetilde{X}-E$, so we obtain $suppT\subseteq E$.

By Theorem \ref{support} for $p= n-1$, we obtain
\begin{displaymath}
 T= \sum_i c_i[E_i],
\end{displaymath}
where $c_i\geq0$ and  $E_i$ are the $(n-1)$-dimensional irreducible components of $E$. Since $T$ is real and d-exact, $i([T]_{BC})=0$ in $H^{1,1}_{A}(\widetilde{X})$. Beacause $i$ is injective, we know $[T]_{BC}=0$ in $H^{1, 1}_{BC}(\widetilde{X})$. So, there is a real $0$-current $Q$ on $\widetilde{X}$, such that $T= i\partial\overline{\partial}Q$. Since $T\geq 0$, $Q$ is plurisubhamonic. By maximum principle, $Q$ is a constant,  hence $T= 0$.
\end{proof}

\begin{lemma}[\cite{FMX}, Lemma 3.6]\label{exact}
Let $f: X\rightarrow Y$ be a modification between reduced compact complex spaces of dimension $n$. If $Y$ is normal and the betti number $b_{2n-1}(Y)= 0$, then there is a exact sequence
\begin{displaymath}
\xymatrix{
0\ar[r] &H_{2n-2}(E, \mathbb{R})\ar[r]^{i_*} &H_{2n-2}(X, \mathbb{R})\ar[r]^{f_*} &H_{2n-2}(Y, \mathbb{R})
}
\end{displaymath}
where $E$ is the exceptional set of $f$, $i: E\rightarrow X$ is the inclusion. Moreover, $H_{2n-2}(E, \mathbb{R})= \oplus_j \mathbb{R}[E_j]$, where $\{E_j\}_j$ are all the $(n-1)$-dimensional irreducible components of $E$ $($possiblly there exist some other components of dimension $< n-1$ in $E$$)$.
\end{lemma}

\begin{theorem}\label{betti}
 If $X$ is a normal strongly Gauduchon space of dimension $n$ with the betti number $b_{2n-1}(X)=0$, then $X\in \mathscr{SG}$.
\end{theorem}
\begin{proof}
Suppose $T$ is a positive $(1, 1)$-current on $\widetilde{X}$ which is $d$-exact. As the proof in Theorem \ref{relation}, we obtain
\begin{displaymath}
 T= \sum_i c_i[E_i]
\end{displaymath}
where $c_i\geq0$, $E_i$ are the $(n-1)$-dimensional irreducible components of $E$. Since $T$ is $d$-exact, $\sum_i c_i [E_i]= [T]_{\widetilde{X}} = 0$ in $H_{2n-2}(\widetilde{X}, \mathbb{R})$. By Lemma \ref{exact}, we get $c_i= 0$ for all $i$.
\end{proof}

\begin{theorem}\label{exc-relation}
Let $X$ be a compact strongly Gauduchon space. If it has a desingularization  $\widetilde{X}$ whose exceptional set has codimension $\geq 2$,
then $X\in \mathscr{SG}$.
\end{theorem}
\begin{proof}
Suppose $\dim X=n$ and $T$ is a positive $(1, 1)$-current on $\widetilde{X}$ which is $d$-exact. As the proof in Theorem \ref{relation}, we obtain $suppT\subseteq E$. By Theorem \ref{support} for $p=n-1$, we get $T= 0$ immediately.
\end{proof}

\section{Families  of complex spaces over a nonsingular curve}

\quad In this section, we study families of complex spaces over a curve. It should be useful in the study of deformations and moduli spaces of complex spaces. The following definition is a generalization of the corresponding notion defined in \cite{M}.

\begin{definition} Let $X$ be a reduced compact complex space of pure dimension $n$, and $f:X\rightarrow C$ a holomorphic map onto a nonsingular compact complex curve $C$. $f$ is called \textup{topologically essential}, if for every $p\in C$, no linear combination $\sum_jc_j[F_j]$ is zero in $H_{2n-2}(X, \mathbb{R})$, where the ${F_j}^,s$ are all the irreducible components of the fibre $f^{-1}(p)$, $c_j\geq 0$ and at least one of the ${c_j}^,s$ is positive.
\end{definition}

Note that, for any  reduced compact complex space $X$ of pure dimension $n$ and the holomorphic map $f:X\rightarrow C$ onto a nonsingular compact complex curve $C$, $f$ is an open map by the open mapping theorem (\cite{GrRe}, page 109). Hence for every $p\in C$, every irreducible component of  $f^{-1}(p)$ has dimension $n-1$ (\cite{Fis}, \S3.10, Theorem).

Now, we can generalize \cite{X}, Theorem 4.1 as follows.
\begin{theorem}\label{topess}
Suppose $X$ is a purely $n$-dimensional compact normal complex space which admits a topologically essential holomorphic map $f:X\rightarrow C$ onto a nonsingular compact complex curve $C$, and $X$ has a desingularization $\pi:\widetilde{X}\to X$, such that no nonzero nonnegative linear combination of hypersurfaces contained in the exceptional set of $\pi$ is zero in $H_{2n-2}(\widetilde{X}, \mathbb{R})$. If every nonsingular fiber of $f$ is a strongly Gauduchon manifold, then $X\in \mathscr{SG}$.
\end{theorem}
\begin{proof}
Set $\tilde{f}:=f\circ\pi$. For every $p\in C$, set $f^{-1}(p)=\bigcup_iV_i$, where $V_i$ are all the irreducible components of $f^{-1}(p)$ which have dimension $n-1$. Since $X$ is normal, $codimX_s \geq 2$, where $X_{s}$ is the set of singular points of $X$. So
\begin{displaymath}
\pi^{-1}(V_i)= \widetilde{V_i}\cup\bigcup_j E_{ij},
\end{displaymath}
where $\widetilde{V_i}=\overline{\pi^{-1}(V_i-X_s)}$ is the strict transform of $V_i$, and $E_{ij}$ are all irreducible components of $\pi^{-1}(V_i)$ contained in the exceptional set of $\pi$. It is possible that some $E_{i j}$ are contained in other $E_{k l}$ or $\widetilde{V_k}$. We denote any $E_{i j}$, which is not properly contained in other $E_{k l}$ or $\widetilde{V_k}$, by $E_{i j'}$  and we denote any $E_{i j}$, which is properly contained in other $E_{k l}$ or $\widetilde{V_k}$, by $E_{i j''}$ (i.e. there exists other $E_{k l}$ or $\widetilde{V_k}$, such that $E_{i j''}\subsetneqq E_{k l}$ or $\widetilde{V_k}$), then
\begin{displaymath}
\tilde{f}^{-1}(p)=\bigcup_i (\widetilde{V_i}\cup \bigcup_{j'} E_{i j'})
\end{displaymath}
is the irreducible decomposition of $\tilde{f}^{-1}(p)$, hence $\textup{codim}E_{i j'}=1$.

We need the following two claims.
\vspace{2mm}

\noindent \textbf{Claim 1.} $\tilde{f}$ is topologically essential.
\begin{proof}
If not, we have
\begin{displaymath}
\Sigma_ia_i[\widetilde{V_i}] + \Sigma_{ij'}b_{ij'}[E_{ij'}] = 0,
\end{displaymath}
in $H_{2n-2}(\widetilde{X}, \mathbb{R})$, for some $a_i$, $b_{ij'}$ $\geq 0$ and at least one of the ${a_i}^,s$, ${b_{ij'}}^,s$ is positive. Since $\pi(E_{ij'}) \subseteq X_s$ has codimension $\geq 2$, $\pi_*[E_{ij'}] = 0$ in $H_{2n-2}(X, \mathbb{R})$. In $H_{2n-2}(X, \mathbb{R})$, $\pi_*[\widetilde{V_i}] = [V_i]$ ,  hence
\begin{displaymath}
\Sigma_ia_i[V_i]=0
\end{displaymath}
through $\pi_*$. Since $f$ is topologically essential, $a_i=0$ for all $i$. So
\begin{displaymath}
\Sigma_{ij'}b_{ij'}[E_{ij'}] = 0,
\end{displaymath}
in $H_{2n-2}(\widetilde{X}, \mathbb{R})$, where $b_{ij'}$ $\geq 0$ and at least one of the ${b_{ij'}}^,s$ is positive. It contradicts the assumption on $\widetilde{X}$.
\end{proof}
\vspace{2mm}

\noindent \textbf{Claim 2.} For every $p\in C$, if $\tilde{f}^{-1}(p)$ is nonsingular, then it is a strongly Gauduchon manifold.
\begin{proof}
Since $\tilde{f}^{-1}(p)=\bigcup_i(\widetilde{V_i}\cup\bigcup_{j'} E_{ij'})$ is nonsingular, we have
\begin{displaymath}
\widetilde{V_i}\cap \widetilde{V_k} = \emptyset, \quad \forall i\neq k;
\end{displaymath}
\begin{displaymath}
\widetilde{V_i}\cap E_{k l'}= \emptyset, \quad \forall i, k, l'.
\end{displaymath}
Since for any $i, j$, $E_{ij}$ is contained in some $E_{k l'}$ or $\widetilde{V_k}$, we have $\widetilde{V_i}\cap E_{ij}=\emptyset$.
On the other hand, if $V_i\cap X_s\neq\emptyset$, then the intersection of $\widetilde{V_i}$ and $\cup_jE_{ij}$ is not empty, which contradicts with $\widetilde{V_i}\cap E_{ij}=\emptyset$. So for all $i$, $V_i\cap X_s=\emptyset$. Hence, the map
\begin{displaymath}
\pi\mid_{\tilde{f}^{-1}(p)}:\tilde{f}^{-1}(p)\rightarrow f^{-1}(p)
\end{displaymath}
is an isomorphism. Since every nonsingular  fiber of $f$ is a strongly Gauduchon manifold and $\tilde{f}^{-1}(p)$ is nonsingular, $\tilde{f}^{-1}(p)$ is a strongly Gauduchon manifold.
\end{proof}

Now, by the Claim 1 and 2, $\widetilde{X}$ is a strongly Gauduchon manifold according to \cite{X}, Theorem 4.1. Hence, $X$ $\in$ $\mathscr{SG}$.
\end{proof}

By the above theorem, we have the following corollary immediately.

\begin{corollary} Suppose $X$ is a purely dimensional compact normal complex space which admits a topologically essential holomorphic map $f:X\rightarrow C$ onto a nonsingular compact complex curve $C$, and $X$ has a desingularization $\widetilde{X}$ whose exceptional set has codimension $\geq 2$.  If every nonsingular  fiber of $f$ is a strongly Gauduchon manifold, then $X\in \mathscr{SG}$.
\end{corollary}

\begin{corollary}
Let $X$ be a purely $n$-dimensional normal compact complex space which admits a topologically essential holomorphic map onto a nonsingular compact complex curve. If the betti number $b_{2n-1}(X)= 0$, then $X\in \mathscr{SG}$.
\begin{proof}
By Lemma \ref{exact}, we know that, for any desingularization $\pi:\widetilde{X}\to X$, $\{[E_j]\}_j$ are linearly independent in $H_{2n-2}(\widetilde{X}, \mathbb{R})$, where  $\{E_j\}_j$ are all the $(n-1)$-dimensional irreducible components of the exceptional set of $\pi$. Using Theorem \ref{topess}, we get this corollary immediately.
\end{proof}
\end{corollary}

\noindent{\bf Acknowledgements.}\,
We would like to thank our supervisor Prof. Jixiang Fu for his constant encouragement
and many helpful discussions.

\end{document}